\documentclass[11pt, reqno]{amsart}
\usepackage[colorlinks=true,pagebackref,hyperindex]{hyperref}
\usepackage{mathrsfs} 
\usepackage[all]{xy}
\usepackage{color}
\usepackage{amsfonts}
\usepackage{amscd}
\usepackage{amssymb,latexsym,amsmath,amscd, graphicx} %Mircea's
\usepackage{bbm}

\definecolor{darkgreen}{rgb}{0.0, 0.4, 0.0}

\definecolor{cyan}{cmyk}{1,0,0,0}

\newcommand{\cm}{\color{magenta}}

\newcommand{\bdg}{\begin{dg}}
\newcommand{\edg}{\end{dg}}

\newtheorem{tm}{Theorem}[subsection]
\newtheorem{lm}[tm]{Lemma}
\newtheorem{pr}[tm]{Proposition}
\newtheorem{rmk}[tm]{Remark}
\newtheorem{cor}[tm]{Corollary}
\newtheorem{ex}[tm]{Example}

\newtheorem{??}[tm]{Question}

\newtheorem{setup}[tm]{Set-up}

\newcommand{\ben}{\begin{enumerate}}
\newcommand{\een}{\end{enumerate}}
\newcommand{\bit}{\begin{itemize}}
\newcommand{\eit}{\end{itemize}}
\newcommand{\beq}{\begin{equation}}
\newcommand{\eeq}{\end{equation}}
\newcommand{\la}{\label}

\newcommand\ci{\cite}

\font\tenmsb=msbm10
\font\sevenmsb=msbm7
\font\fivemsb=msbm5
\newfam\msbfam
\textfont\msbfam=\tenmsb
\scriptfont\msbfam=\sevenmsb
\scriptscriptfont\msbfam=\fivemsb
\def\Bbb#1{{\fam\msbfam #1}}
\font\teneufm=eufm10
\font\seveneufm=eufm7
\font\fiveeufm=eufm5
\newfam\eufmfam
\textfont\eufmfam=\teneufm
\scriptfont\eufmfam=\seveneufm
\scriptscriptfont\eufmfam=\fiveeufm

\newcommand{\ke}{ \hbox{\rm Ker} }

\newcommand\rat{{\Bbb Q}}

\newcommand\oql{\overline{\Bbb Q}_\ell}
\newcommand\comp{{\Bbb C}}

\newcommand\zed{{\Bbb Z}}

\newcommand\Om{\Omega}

\newcommand\e{\epsilon}

\newcommand{\w}[1]{\widetilde{#1}}

\newcommand{\ov}[1]{\overline{#1}}

\newcommand{\m}[1]{\mathcal{#1}}

\newcommand{\hm}{M}
\newcommand{\ha}{A}

\newcommand{\hh}{h}

\newcommand{\Gm}{{\mathbb G}_m}

\newcommand{\X}{X}
\newcommand{\SP}{S}

\newcommand{\fa}{\X/S}
\newcommand{\G}{G}

\newcommand{\md}[2]{M_D({#1},{#2})}

\newcommand{\ma}[2]{A({#1},{#2})}
\newcommand{\h}{h}

\title{Projective compactification of \\ Dolbeault moduli spaces}
\author{
Mark Andrea A.  de Cataldo}
\address{Mark Andrea A. de Cataldo, Stony Brook University}
\email{mark.decataldo@stonybrook.edu}

\begin{document}

\begin{abstract}
We construct a relative projective compactification of  Dolbeault moduli spaces of Higgs bundles  for reductive algebraic groups 
on families of projective manifolds   that  is compatible with the Hitchin morphism.
\end{abstract}

\maketitle

\tableofcontents

\section{Introduction}\la{intro}

The purpose of this paper is to prove Theorem \ref{sicomp}, which provides a natural  projective compactification of Simpson's Dolbeault moduli spaces of Higgs bundles for complex reductive algebraic groups on projective manifolds. The compactification statement seems to be folklore, but we could not locate a reference in the literature. The projectivity assertion seems new.
Remark \ref{other} discusses the earlier work we are aware of; 
\S\ref{asco} discusses in more detail the relation of this work to the work of A. Schmitt.
In the course of proving our main result, we  establish some complements which can be of independent interest.
Next, we discuss in more detail the contents of this paper.

The Dolbeault moduli space for a reductive algebraic group $G$ for a family $\fa$    of projective manifolds 
 is quasi projective over the base $\SP$.
The associated Hitchin morphism is proved to be proper in the case $G=GL_n$ by Simpson. 

We observe in Proposition
\ref{hitprop} that
the Hitchin morphism is proper, in fact projective, for every reductive algebraic group $G$. The properness assertion has been independently proved for families of curves in arbitrary characteristic in 
\ci{ahh}. The target of the Hitchin morphism is a global version for the family $\fa$ of the
quotient $\mathbbm g /\!/ G$. In this context, the Chevalley restriction morphism being an isomorphism
plays an important role, albeit not a direct one in this paper. Since we could not locate in the literature a
reference for this fact in the case $G$ reductive algebraic, we offer a proof in Lemma \ref{crm}. Our proof of the properness of the  Hitchin morphism consists of exhibiting it as the first link in a  factorization of another proper morphism. Since the second link is of great Lie-theoretic importance, we show it  is  a finite morphism
in Proposition \ref{gcrt}.

Proposition \ref{rg90} constructs a natural complex on the Dolbeault moduli space that, locally over the base $\SP$, is the box product of the intersection complex of a typical fiber (via the Non Abelian Hodge Theorem,
the Dolbeault moduli space is topologically locally trivial over the base) with the constant sheaf over the base.
Once this is done, the last assertion of the proposition, i.e. the vanishing  $\phi F =0$ of the vanishing cycle, 
follows directly.

The main  result of this paper is the compactification  Theorem \ref{sicomp}, the proof of which is spread-out through several subsections
of \S\ref{doe45}. 
We use Simpson's compactification  Theorem \ref{siomp} in the context of suitable $\Gm$-actions, of which we need the  amplification provided by Proposition 
\ref{sicompv};  this slight improvement  also allows to incorporate the Hitchin morphism in the compactification framework.
\S \ref{achyt} constructs  the desired  compactification.
Away from the nilpotent cone, i.e. the fiber of the Hitchin morphism over the unique  $\Gm$-fixed point of the Hitchin base, the stabilizers  of the natural $\Gm$-action on the Dolbeault moduli space are finite;
when the Dolbeault moduli space is an orbifold (this is rare, but it happens in very interesting cases; see Remark \ref{sm}), Lemma \ref{orb} allows to deduce that the compactification is an orbifold as well; in this context, we could not locate in the literature a needed technical statement, hence the lemma, which was suggested to us by M. Brion.
\S\ref{dcv} contains the proof of our main Theorem \ref{sicomp}. Proposition \ref{vcb1} contains
some topological complements that our compactification affords when the Dolbeault moduli space is an orbifold (cf. Theorem \ref{sicomp}.(6)).

 As it is pointed out in Remark \ref{H moduli space},  Theorem \ref{sicomp}, parts 1-5 holds in the more general context of $\Lambda$-modules,
with $\Lambda$ of polynomial type. The proofs are identical. The case of Higgs sheaves is then a special case, and the one of Dolbeault moduli spaces is an even more special case. We have decided to write this paper in the context of Dolbeault moduli spaces because of the extra appeal stemming from the Non Abelian Hodge Theorem.

{\bf Acknowledgments.}
The author  thanks: Leticia Brambila Paz, Michel Brion, Victor Ginzburg, Jochen Heinloth, Andrea Maffei, Luca Migliorini, Mircea Musta\c t\u a,  Daniel Bergh,  J\"org Sch\"urmann , Carlos Simpson and Geordie Williamson
for useful conversations. The  author, who is partially supported by N.S.F. D.M.S. Grants n. 1600515 and 1901975, would also like to thank the Freiburg Research Institute for Advanced Studies  for the perfect working conditions; the research leading to these results has received funding from the People Programme (Marie Curie Actions) of the European Union's Seventh Framework Programme
(FP7/2007-2013) under REA grant agreement n. [609305].    I thank the anonymous referees for their suggestions, especially in connection with
their constructive criticism of the original proof of Theorem \ref{sicomp}.(1,2) (projectivity of the compactification).

 %\newpage
\subsection{Notation}\la{notation}$\;$

We work over the field of complex numbers $\comp$. A variety is a separated scheme of finite type
over $\comp$. All varieties in this paper turn out to be quasiprojective over any chosen base variety $S$.

A standard reference for Higgs bundles and Dolbeault moduli spaces is \ci{si1,si2}.
For the derived category of constructible sheaves, we refer the reader to \ci{bams}.
For vanishing cycles,  we refer the reader to \ci{dema}.

%\newpage
\section{Dolbeault moduli spaces: review and complements}\la{rcdm}

In this section, we review Simpson's Dolbeault moduli spaces.
The main reference for this section is  \ci{si2} where, among other things, C. Simpson proves the Non Abelian Hodge Theorem in families over a base $S$.  This section also contains some folklore complements that do not seem to be   documented in the literature we are aware of: projectivity of the Hitchin morphism for reductive algebraic groups (Proposition \ref{hitprop}); the Chevalley restriction isomorphism  (Lemma \ref{crm}); the finiteness assertion of Proposition \ref{gcrt}.  Proposition \ref{rg90} constructs a complex on the Dolbeault moduli space for a  family of projective manifolds  that restricts to the intersection cohomology complexes on the fibers; this seems new.

In this section, we place ourselves in the following: 
\begin{setup}\la{setup}
Let $\G$ be a complex reductive algebraic group. 
Let $\fa$ be a smooth projective morphism (family).  
%Let $\m{O}_{\X}(1)$ be  a  line bundle on $\X$ which is very ample on every member of the family $\fa$. 
\end{setup}

Given a point  $s \in \SP$, we denote   by $\X_s$ the corresponding member of the family. 
More generally, a subscript
$-_s$, with $s\in S$, indicates the restriction of an object to the corresponding fiber.

%{\cm Clarify: role of $\m{O}_X(1)$, esp. since it does not seem to appear in Betti ...}

%\newpage
\subsection{The Dolbeault moduli space}\la{dms}$\;$

Let $\md{\fa}{\G}/\SP$ be the  relative Dolbeault moduli space associated with the reductive algebraic group $\G$ and  the family $\fa$,
and let:
\beq\la{dmsp}
\xymatrix{
 \pi_D(\fa, \G): \md{\fa}{\G} \ar[r] & \SP
}
\eeq
 be the structural morphism.  This moduli space universally corepresents the appropriate functor.
If $s \in \SP$, then the fiber $\pi_D(\fa,\G)^{-1}(s)$
 is the Dolbeault moduli space  $\md{\X_s}{\G}$ associated with $G$ and  $\X_s$. 

%The line bundle $\m{O}_{\X}(1)$ is used by Simpson to construct various associated moduli spaces; however, it turns out that  the Dolbeault moduli space is independent of this choice. 

For the case $\G=GL_n$, see \ci[pp.16-17]{si2} and \ci[Theorem 4.7]{si1}; the Dolbeault moduli space is obtained
as a good quotient of a parameter space $Q$ by the action of a special linear group. The morphism $\pi_D(\fa, GL_n)$ is quasi-projective (cf.
\ci[Theorem 4.7]{si1}),
and the  closed points in  $\md{\X_s}{GL_n}$ parameterize Jordan equivalence classes of $\mu$-semistable Higgs
bundles of rank $n$ on $\X_s$ with vanishing rational Chern classes $c_i$, $\forall i >0$; (see \ci[p.17]{si2}).
There is also the construction stemming from \ci[Proposition  9.7]{si2} and \ci[Theorem 4.10]{si1}, where the Dolbeault moduli space arises in connection with good and geometric quotients of  Dolbeault representation spaces modulo the action of $GL_n$; this construction
is used in the construction of moduli spaces for $\G$ reductive \ci[Proposition  9.7]{si2},  via a closed embedding of $\G$ into some $GL_n$.
There is also the construction relating Higgs sheaves to sheaves in the relative cotangent bundle \ci[p.18]{si2}, which Simpson
uses to prove the properness of the Hitchin morphism for $\G=GL_n$.

For $G$ reductive algebraic, the morphism $\pi_D(\fa,\G)$ is again quasi-projective: combine \ci[Proposition 9.7]{si2}, \ci[Corollary 9.19]{si2}
and
\ci[\href{https:/\!/stacks.maTheoremcolumbia.edu/tag/0417}{Tag 0417, Pr. 58.49.2}]{stacks-project}.  The closed points  in $\md{\X_s}{\G}$ 
parameterize the set of isomorphism classes of principal Higgs bundles of semiharmonic type on $\X_s$ for the reductive algebraic  group $G$ (cf. \ci[Proposition 9.7]{si2}).

\begin{rmk}\la{var}
{\rm ({\bf Higgs vector bundles over curves})}
If $\fa$ is a family of smooth projective curves of genus $g \geq 2$ and  $\G=GL_n$,
then, fiberwise over $\SP$,  the Dolbeault moduli spaces $\md{\X_s}{\G}$  are integral and normal
see \ci[Corollary 11.7]{si2}.
\end{rmk}

The Dolbeault moduli spaces of a smooth projective variety  are seldom nonsingular: the only case I know of is
the case $\G=GL_1$, where the moduli space is the cotangent bundle to ${\rm Pic}^0$.

\begin{rmk}\la{sm} {\rm ({\bf Variant: Higgs vector bundles   over curves with   degree coprime to the rank})}
 The following variant of Dolbeault  moduli spaces are nonsingular and 
connected, moreover,  the analogue of the Non Abelian Hodge Theorem holds for them:
$\fa$ is a family of  projective connected nonsingular \underline{curves} of genus $g \geq 2$, the reductive algebraic group  $\G=GL_n, SL_n$, and we consider 
stable Higgs bundle of degree coprime to the rank. For $\G=PGL_n$ one gets the quotient
of the $SL_n$-moduli space by the  abelian group scheme ${\rm Pic}^0_{\fa}[n]$,  which is finite over $\SP$.  See \ci{dhm} and the references therein, and \ci[\S6]{si}.
\end{rmk}

\subsection{Projectivity of the Hitchin morphism}\la{dms11}$\;$

When $\G=GL_n$,
the Hitchin morphism 
\beq\la{eqhm}
\xymatrix{
\h (\fa,\G): \md{\fa}{\G} \ar[r] &  \ma{\fa}{\G}
}
\eeq
is defined   in \ci[p.22]{si2}.  Here, $\ma{\fa}{GL_n}$ is the scheme representing the functor sending
$S'/S$ to $\oplus_{i=1}^n H^0 (X':=\X \times_\SP \SP', {\rm Sym}^i \, \Om^1_{\X'/\SP})$; according to general facts, this representing scheme is a cone ${\rm Spec}_{\m{O}_\SP} (\m{Q})$ over $\SP$, for a suitable coherent $\m{O}_\SP$-module
$\m{Q}=\oplus_i \m{Q}_i$ (e.g. cf.  \ci[Lemma 3.1.3]{wang}).  In short: first, one chooses
a homogeneous system of generators $(f_i)_{i=1}^n \subseteq \comp [\mathbbm g \mathbbm l_n /\!/ G] = \comp [\mathbbm g \mathbbm l_n]^{GL_n}\subseteq  \comp [\mathbbm g \mathbbm l_n]$  of degree $i$, e.g.
${\rm trace}( \wedge^i(-))$; then, given
a Higgs bundle $(E, \phi)$ on $X'/\SP'$, one combines the  $f_i$ with the twisted endomorphism $\phi$  to define the sections of 
${\rm Sym}^i \, \Om^1_{\X'/\SP}$. 
%One may avoid the choices made above (and the ones made below 
%for the reductive algebraic $\G$) by employing $\mathbbm g /\!/ \G$, which is non-canonically isomorphic to the spectrum of a polynomial $\comp$-algebra.
In the case where $G$ is reductive algebraic, one defines the Hitchin morphism in the same way,
by  choosing a homogeneous system of generators $f_j  \in \comp [\mathbbm g]^{\G}\subseteq  \comp [\mathbbm g]$
with degrees $d_j$ given by the fundamental degrees of $\mathbbm g$.

\begin{rmk}\la{ig}
Note that  for $s \in \SP$, we have that
$\ma{\fa}{\G}_s= \ma{\X_s}{\G}$.
Let $\SP$ be connected. When $\dim{X/S} \geq 2$,   the dimensions $h^0 (\X_s, {\rm Sym}^i \Om^1_{\X_s})$ may jump up; see \ci[\S4.2]{Ro-Ro}.
These dimensions do not  jump  when $\dim{X/S}=1$, i.e. for families of curves.  Regardless of the relative dimension $\dim{X/S}$ of the family $X/S$, by Hodge Theory, these dimensions do not jump  for $i=1$. 

\end{rmk}

\begin{pr}\la{hitprop}  {\rm ({\bf Projectivity of the Hitchin morphism (\ref{eqhm})})}
The Hitchin morphism (\ref{eqhm}))
 is  projective.
\end{pr}
\begin{proof}
Since the Dolbeault moduli space structural morphism (\ref{dmsp}) is quasi-projective,
it is enough to prove that the Hitchin $\SP$-morphism (\ref{eqhm}) is proper.

 In  the case  $\G=GL_n$, properness of the Hitchin morphism follows from
\ci[Theorem 6.11]{si2}. In the case when $\G$ is reductive algebraic, we argue as follows. We first embed $\G$ into some $GL_n$ as a closed subgroup. We thus obtain 
the commutative diagram of $\SP$-morphisms:
\beq\la{cdhm}
\xymatrix{
\md{\fa}{\G}    \ar[r]^-{\iota_M} \ar[d]^-{\h(\fa, \G)} & \md{\fa}{GL_n}  \ar[d]^-{\h(\fa, GL_n)}   \\
\ma{\fa}{\G}    \ar[r]^-{\iota_A}  & \ma{\fa}{GL_n}. 
}
\eeq
The morphism $\h(\fa, GL_n) \circ  \iota_M$ is proper (cf. \ci[Theorem 6.11 ($\h$ is proper) and Corollary 9.15 ($\iota_M$ is proper), or the more precise  Corollary 9.19 ($\iota_M$ is finite)]{si2}). It follows that
the morphism $\h(\fa, \G)$ in (\ref{eqhm})   is proper, as predicated. 
\end{proof}

\begin{rmk}\la{ggtt}
The case of $\G=GL_n$ and families  $\fa$ of arbitrary relative dimension  is due to C. Simpson \ci{si2}. Proposition \ref{hitprop}  is a simple complement to Simpson's proof of properness. The case of $G$-semisimple for families of curves is due to G. Faltings \ci[Theorem I.3]{fa}.
The paper \ci{ahh} contains a proof of properness for  $G$ reductive algebraic for families of curves in arbitrary characteristic.
\end{rmk}

\begin{rmk}\la{itsfin} {\rm ({\bf Complement: $\iota_A$ is finite})}
We observe that, as one may expect, the morphism  $\iota_A$ is finite.  Since we could not locate a reference in the literature for this seemingly well-known fact, we offer a proof in the slightly more general  Proposition \ref{gcrt}.
\end{rmk}

\subsection{$\Gm$-equivariance of the Hitchin morphism}\la{dms22}$\;$

The group $\Gm$ acts on the Hitchin $\SP$-morphism (\ref{eqhm}) as follows. It acts trivially on $\SP$.
it acts on the Dolbeault moduli space   by scalar multiplication of   Higgs fields (these are suitable sections of a twisted adjoint bundle; cf \ci[p.49]{si2}), and  this action covers the
trivial action over $\SP$; see {\cm \ci[p.17-18 and p.62]{si2}}. It acts on the Hitchin base with positive weights $d_j$: let $\m{Q}=\oplus_j \m{Q}_j$ be the coherent sheaf on $\SP$
whose symmetric $\m{O}_\SP$-algebra has spectrum representing the Hitchin base (cf. \S\ref{dms11}); 
we view it as a graded $\m{O}_S$-algebra by setting $\deg \m{Q}_j:=d_j$; then $t\in \Gm$ acts by multiplication by $t^{d_i}$ on each 
$\m{Q}_i$. This action also covers the trivial action
over $\SP$. The Hitchin morphism  is $\Gm$-equivariant for the aforementioned actions.
Diagram (\ref{cdhm}), is a diagram of $\Gm$-equivariant $\SP$-morphisms.

\begin{rmk}\la{wps}
{\rm {\bf (Weighted projective completion of the Hitchin base)}}
For cones and their projective completions, see \ci[Appendix B5]{fu}.
The Hitchin base $\ma{\fa}{\G}$ is the  affine cone ${\rm Spec}_{{\m{O}_\SP}} (\m{S}:={\rm Sym}^\bullet (\m{Q}))$ over $\SP$. By taking {\rm Proj} of the associated graded algebra $\m{S}[z]$, with $z$ a free variable of weight $1$, and  with remaining weigths $d_j$ as   specified above, 
we obtain the relative projective cone completion $\overline{\ma{\fa}{\G}}$ over $\SP$ of the Hitchin base.
Its fibers over $\SP$ are wighted projective spaces, with dimensions varying upper-semicontinuously. 
It carries a natural $\Gm$-action, compatible with the one on the affine cone. The divisor at infinity is Cartier and made of $\Gm$-fixed points.
\end{rmk}

\begin{rmk}\la{lim}
{\rm {\bf (Existence of zero-limits for the $\Gm$-action)}}
An important consequence of the properness of the Hitchin morphism  is that, since the $\Gm$-action on the target of the Hitchin morphism is contracting,
the zero-limits for the $\Gm$-action exist  in the Dolbeault moduli spaces $\md{\X_s}{\G}$; these limits  are fixed points and they  
dwell in the fiber of the Hitchin morphism over the origin (the unique fixed point) of the target. See \ci[Corollary 9.20]{si2}.
This important consequence allows us to use Simpson's compactification technique, amplified in Proposition
\ref{sicompv}, in the proof of our compactification Theorem \ref{sicomp}.
\end{rmk}

\begin{rmk}\la{e2d}
{\rm {\bf (Cohomological consequences for contracting $\Gm$-actions)}}
A well-known  consequence of the properness of the Hitchin morphism  coupled with  the fact that
the   $\Gm$-action on the target of the Hitchin morphism is contracting is that, given a projective manifold $X$, we have natural isomorphisms:
 \[H^*(\md{X}{\G}, \rat)=H^*({\h(X, \G)^{-1}(o),\rat)}\] between the rational cohomology groups of the Dolbeault moduli space and the one of the fiber of the Hitchin morphism over the origin (nilpotent cone);
 the same holds for $I\!H^*(\md{X}{\G}, \rat)=H^*({\h(X, \G)^{-1}(o),(\m{IC}_{\md{X}{\G}, \rat})_{|
 {\h(X, \G)^{-1}(o)}}})$.  See \ci[Lemma 6.11 and Remark. 6.12]{dehali}. In fact, the corresponding Leray spectral sequences are $E_2$-degenerate, and their $E_2$-pages consist of only one non-zero column, i.e. $E_2^{0q}$.
\end{rmk}

%\newpage
\subsection{Complement: minor variation on the Chevalley restriction theorem}\la{vic}$\;$

The standard formulation of Chevalley's restriction theorem that we have been able to locate in the literature  \ci{hum, ch-gi} is as follows:
let $G$ be a complex connected semisimple algebraic group, let $\mathbbm g$ be the Lie algebra of $G$, let $\mathbbm h$
be a Cartan subalgebra of $\mathbbm g$, let $W$ be the associated Weyl group; then the natural  Chevalley restriction morphism
$\mathbbm g /\!/G \to \mathbbm h/W$ is an isomorphism. 
It is well-known that the same conclusion holds if we replace $G$ semisimple and connected
with $G$ reductive algebraic. I thank V. Ginzburg and G. Williamson for suggesting us how to prove it.
I thank an anonymous referee for pointing us to the fact that what follows can also  be deduced from \ci[Theorem 4.2]{luri}.

\begin{lm}\la{crm}
Let $G$ be connected and reductive algebraic. Then the Chevalley restriction morphism $\mathbbm g /\!/G \to \mathbbm h/W$ is an isomorphism.
\end{lm}
\begin{proof}
Let $Z^o\subseteq G$ be the identity component of the center of $G$: it coincides with the radical of $G$ and it is   a torus \ci[Thm. 5.1]{milneAGS}. The quotient group $q: G\to G':=G/Z^o$ is semisimple \ci[Pr. 2.5]{milneAGS}. 

Let $\mathbbm g$ denote
the Lie algebra of $G$. Similarly, we have $\mathbbm g'$ for $G'$, and $\mathbbm z$ for $Z^o$. 

Let $T'\subseteq G'$ be a maximal torus and let $\mathbbm t'$ be its Lie algebra. Then we have the root-space decomposition
of $T'$-modules $\mathbbm g' = \mathbbm t' \oplus \mathbbm r'$.

The Lie algebra  $\mathbbm g$ is a $G$-module for the (adjoint) action of $G$
and, since $Z^o$ acts trivially on $\mathbbm g$, the $G$-action factors through $G'$. 

Since $G$ is reductive algebraic, we have a non-canonical splitting  $\mathbbm g= \mathbbm z \oplus \mathbbm g_1$ of $G$-modules, as well as of $G'$-modules.

The differential $dq: \mathbbm g = \mathbbm z \oplus \mathbbm g_1 \to \mathbbm g'$ is $G'$-equivariant and split, and it identifies 
$\mathbbm g_1$ and $\mathbbm g'$ as $G'$-modules.

Let $T\subseteq G$ be the  pre-image of $T'$. It is  a maximal torus; this can be seen, for example, as a consequence of  the Hofmann-Schereer Splitting Theorem, which, in particular, says that the commutator subgroup intersects trivially a suitable torus, therefore it intersects trivially the center; this implies $T$ is commutative and, since $Z^o$ is connected, it is also connected; one then shows it is a torus, and a maximal one, non-canonically isomorphic to $Z^o\times T'$. 

We note that the Weyl group $W:=W(G,T)=N(T)/T$ maps isomorphically, via $q$, onto $W':= W(G',T')=N(T')/T'$.

We thus have the natural commutative diagram of $\comp$-algebras:
\beq\la{rffrrf}
\xymatrix{
\comp [\mathbbm g']   \ar[r] \ar[d]  & \comp [\mathbbm t'] \ar[d] \\
\comp [\mathbbm g]   \ar[r] & \comp [\mathbbm t],
}
\eeq
where: the horizontal maps are given by restrictions of functions;  the vertical ones are pull-backs of functions and, the l.h.s. one is  a morphisms of  $G'$-modules, the r.h.s. one is a morphisms of $W=W'$-modules.

 By taking invariants in
(\ref{rffrrf}), we have  the natural commutative diagram of $\comp$-algebras:
\beq\la{rff}
\xymatrix{
\comp [\mathbbm g']^{G'}   \ar[r]^-\cong \ar[d]^-{1 \otimes -}  & \comp [\mathbbm t']^{W'} \ar[d]^-{1 \otimes -} \\
\comp [\mathbbm z] \otimes_\comp \comp [\mathbbm g_1]^{G'}=  \comp [\mathbbm g]^{G'} =\comp [\mathbbm g]^{G}   \ar[r]^-{1\otimes -}_-\cong & 
\comp [\mathbbm z] \otimes_\comp \comp [\mathbbm t'_1]^{W'}
=
\comp [\mathbbm t]^{W'}=  \comp[\mathbbm t]^{W},
}
\eeq 
where:  the top horizontal arrow is the Chevalley Restriction Isomorphism for the semisimple $G'$;
the vertical arrows map a function $f$ to $1\otimes f$; the identifications on the bottom horizontal arrows
follow from the splitting of modules constructed above; the bottom horizontal arrow is the tensor product of
the identity on $\comp [\mathbbm z]$ with the Chevalley restriction morphism for $G'$, via the identifications given above.

The desired conclusion  follows. 
\end{proof}

The following was suggested to us by J. Heinloth.  Since we do do not know of a reference, we offer a proof. We thank V. Ginzburg for suggesting the one below.
We thank T. Haines, J. Heinloth,  and J.E. Humphreys for helpful discussions.

\begin{pr}\la{gcrt}
Let $G \to M$ be a finite morphism of complex reductive algebraic Lie groups. Then the natural morphism induced by the adjoint actions:
\beq\la{efin}
\xymatrix{
\mathbbm g /\!/G  \ar[r] & \mathbbm m /\!/M}
\eeq
is finite.
\end{pr}
\begin{proof}
Let $G^o, M^o$ denote the respective connected components of the identity. There is the natural commutative diagram of morphism:
\beq\la{ff1}
\xymatrix{
\mathbbm g /\!/G^o \ar[r] \ar[d] & \mathbbm m /\!/ M^o \ar[d]   \\
\mathbbm g /\!/G \ar[r]  & \mathbbm m /\!/ M,
}
\eeq
where the vertical arrows are finite and surjective. It follows that if  the top horizontal arrow is finite, then so is the bottom one
(properness follows from    surjectivity, and  quasi-finiteness is evident), so that
we may assume that the reductive algebraic groups $G$ and $M$ are connected.

Since the morphism $G\to M$ is assumed to be finite, the differential $\mathbbm g \to \mathbbm m$ is injective
and we may view $\mathbbm g$ as dwelling inside  $\mathbbm m$.

By using the maximality of Cartan subalgebras,  we can choose Cartan subalgebras $\mathbbm h (\mathbbm g)  \subseteq 
\mathbbm h (\mathbbm m).$  Let $W(\mathbbm h (\mathbbm g))$ and $W(\mathbbm h (\mathbbm m))$ be the corresponding Weyl groups.

The morphism $\mathbbm g /\!/ G = \mathbbm h (\mathbbm g) / W(\mathbbm h (\mathbbm g) \to
\mathbbm m /\!/ M = \mathbbm h (\mathbbm m) / W(\mathbbm h (\mathbbm m)$  (cf. Lemma \ref{crm}) is finite because the Weyl groups are finite
and $h (\mathbbm g) \subseteq  h (\mathbbm m).$ 
\end{proof}

\begin{rmk}\la{voo}
Even if the given morphism 
of reductive algebraic groups is a closed embedding, the morphism (\ref{efin}) may fail to be  a closed embedding.
Consider  the classical embedding $SO(4) \subseteq GL_4$ (more generally, $SO(2n) \subseteq GL_{2n})$:
then the algebra of invariants is a polynomial algebra with generators  $s_2, p_2$, where $p_2$ is the Pfaffian and satisfies
  $p_2^2= \det$,  the determinant; it follows that, in this case,  (\ref{efin}) is $2:1$ onto its image.
\end{rmk}

\subsection{Vanishing of vanishing cycles}\la{dms33}$\;$

In general, due to the possible singularity of the base  $\SP$ of the family $\fa$, the intersection complex of the Dolbeault moduli space over  $\SP$ does not restrict to the intersection complexes of the Dolbeault moduli spaces of the fibers over $s \in \SP$. A priori, even if $\SP$ is nonsingular, it is not  immediately clear that there should be a complex on the Dolbeault moduli space $M/\SP$
that restricts to the intersection complexes of the Dolbeault moduli spaces of the fibers over $s \in \SP$. 

 The following proposition  is   an application of the gluing Lemma \ci[Thm. 3.2.4]{bbd}, and  it ensures that
there is a natural complex  on the Dolbeault moduli space   $M/\SP$ which restricts to the intersection complexes of the Dolbeault moduli spaces of the fibers over $s \in \SP$.  The vanishing $\phi F=0$ 
is an amplification of  \ci[Lm. 4.1.9 and Corollary 4.1.4]{dema}.

\begin{pr}\la{rg90}
{\rm ({\bf The  complex  $F$ and the vanishing  $\phi F=0$})}
Let $p: M \to \SP$ be a morphism of varieties  that is topologically locally trivial over the base $\SP$.
Then there is a complex $F \in D(M)$ that, locally over $\SP$,
is a box product of the intersection complex of a typical fiber with the constant sheaf ${\rat_{\SP}}$.
In particular, $F$ restricts to the intersection complexes of the fibers $M_s$ of $M/\SP$. 
If $S$ is a non singular curve and $s\in S$ is a point, then the vanishing cycle complex $\phi F=0$. 
\end{pr}
\begin{proof} 
We may assume that $\SP$ is connected. Let $\m{M}$ be an algebraic variety which is a representative of the homeomorphism
class of the fibers of $M/\SP$. Note that $\m{M}$ is not necessarily irreducible, nor connected; the intersection complex
of such varieties is defined to be the direct sum of the intersection complexes  of its irreducible components as in \ci{dejag}, or in \ci{dema}, where it is also proved that it is a homeomorphism invariant by reducing to the  pure-dimensional case, proved by M. Goresky and R. MacPherson. This is proved directly by Ben Wu in 
\ci{wu}.
Let $\{\SP_a\}$ be an open covering of $\SP$ such that the $\SP_a$ are contractible
and such that $M/\SP$ is trivialized over the open sets $\SP_a$ by means of $\SP_a$-homeomorphisms
 $\phi_a:M_a:=p^{-1} (\SP_a) \stackrel{\sim}\to S_a \times  \m{M}$. Let $q_a: \SP_a \times \m{M} \to \m{M}$ be the projection.
 
 We set $\SP_{ab}= \SP_{a}\cap \SP_{b}=\SP_{b} \cap \SP_a= S_{ba}$. Similarly for triple intersections. We also  have
 $M_{ab}$, $q_{ab}$, etc. We denote the restrictions of the $\phi$'s as follows:
 $\phi_{a|b}:= {\phi_a}_{|M_{ab}}$. We also have the transition $\SP_{ab}$-homeomorphisms 
 $\phi_{ba}:= \phi_{b|a} \circ \phi_{a|b}^{-1}$ and their restrictions, denoted $\phi_{ba|c}$ to  triple intersections. The cocycle identities then read as follows: $\phi_{cb|a} \circ \phi_{ba|c} = \phi_{ca|b}$.

Let $I:= IC_{\m{M}}$ be the intersection complex of $\m{M}$. By \ci[Lm. 4.1.3 and its proof]{dema}, $I$ is perverse semisimple
and it is characterized by the following conditions being met: being perverse semisimple; having its simple summands supported precisely on the  irreducible components of $\m{M}$; being
the direct sum $\oplus_{T^o} \rat_{T^o} [\dim T]$ of the constant sheaves on any Zariski dense open subset $T^o$  the  regular part $T^{\rm reg}$ of each irreducible component $T$ of $\m{M}$, shifted by the dimension of such component.

Let $F_a:= \phi_a^* q_a^*  I \in D(M_a)$. 
We have the chain of canonical identifications: (first ${\rm Hom}$ in $D(M_a)$, the others in $D(\m{M})$):
\beq\la{gl1}
{\rm Hom} (F_a, F_a [i])= {\rm Hom} (I, {q_a}_* {\phi_a}_* \phi_a^* q^*_a I [i]) = {\rm Hom} (I, {q_a}_*  q^*_a I [i])
={\rm Hom} (I,   I [i]),
\eeq
where: the first equality holds by the usual adjunction between pull-back and push-forward; the second equality  holds because
$\phi_a$ is a homeomorphism; the third equality holds is by the Vietoris-Begle Theorem \ci[Pr. 2.7.8]{kash}, in view of the contractibility 
of $\SP_a$. Since $I$ is perverse, we have that the last term in (\ref{gl1}) vanishes $\forall i<0$, and we get that:
\beq\la{gl1b}
{\rm Hom} (F_a, F_a [i])=0, \quad \forall i<0.
\eeq

We denote restrictions as follows $F_{a|b}:= {F_a}_{|M_{ab}}$.
By adjunction again, we have:
\beq\la{gl2}
{\rm Hom} (F_{a|b}, F_{b|a}=  {\rm Hom} (I,  {q_{ab}}_* {\phi_{a|b}}_* \phi_{b|a}^* q_{ab}^* I).
\eeq
By the characterization of $I$, the second argument in the last term is canonically isomorphic to $I$. Let 
$\rho_{ba} \in {\rm Hom} (F_{a|b}, F_{b|a})$ be the element corresponding to this identification via
(\ref{gl2}). It follows that the $\rho$'s  satisfy the cocycle condition.

In view of the glueing lemma \ci[Thm. 3.2.4]{bbd}, we have an object $F$, unique up to unique 
isomorphism, that glues the $F_a$.

The vanishing  $\phi F =0$ follows directly from the local triviality  of $F$ over $S$.
\end{proof}

\begin{rmk}
{\rm ({\bf Twisting by local systems})}
Once we have constructed $F$ as in the proof of Proposition \ref{rg90}, we can twist it by the pull-back
of any rank one local system on $\SP$ and obtain other constructible complexes that restrict to the intersection complexes
of the fibers. These correspond to modified choices of the gluing data given by the  $\rho$'s in the aforementioned proof.
\end{rmk}

\begin{rmk}\la{atvr}
The evident variant of Proposition \ref{rg90} in the context of the twisted Dolbeault spaces of Remark \ref{sm} holds, with the same proof.
\end{rmk}

%\newpage

\section{Projective compactification of Dolbeault moduli spaces}\la{doe45}

We freely use the  the Set-up \ref{setup} and the notation and results in \S\ref{rcdm}.

%\newpage
\subsection{The projective compactification statement}\la{codms}$\;$

Denote the Hitchin $\SP$-morphism (\ref{eqhm}) for the  smooth projective family $\fa$ and the reductive algebraic group $G$ simply by: 
\beq\la{eqhms}
\xymatrix{
\hh: \hm  \ar[r] & \ha.
}
\eeq

The structural $\SP$-morphism for $M/\SP$ is usually not proper: just consider the $\Gm$-action which rescales the Higgs field so that its image under the Hitchin morphism escapes to infinity. It is desirable to produce a compactification of $M$ relative to this morphism that retains many of the properties of $M$, especially in connection with the Hitchin morphism.
 We provide such a compactification in Theorem \ref{sicomp}. In some special cases, this compactification has some precursors; see
 Remark \ref{other}.

When dealing with Cartesian diagrams, we denote parallel arrows with the same symbol. This abuse of notation does not create conflicts in what follows. 

\begin{tm}\la{sicomp} {\rm ({\bf Relative projective compactification of Dolbeault moduli spaces})} 
Let $\fa$ be a smooth projective family, let $G$ be a reductive algebraic group and consider the Hitchin $\SP$-morphism $h$ (\ref{eqhms}).
There is a  Cartesian diagram of $S$-varieties:
\beq\la{sicomp1}
\xymatrix{
Z  \ar[d]^-{\hh} \ar[r]^-a &    \ov{\hm}   \ar[d]^-{\hh}  &    \hm  \ar[d]^-{\hh} \ar[l]_-b 
\\
W   \ar[r]^-a &  \ov{\ha}  &      \ha \ar[l]_-b
}
\eeq
such that:

\ben
\item
The $\SP$-structural morphisms for the varieties in the left-hand and middle columns are projective (in general, $M$ and $A$ are not proper over $\SP$).

\item
The morphisms $\hh$ of Hitchin-type are projective.

\item All morphisms in {\rm (\ref{sicomp1})}, including the omitted structural morphisms to $\SP,$  are $\Gm$-equivariant morphisms of $\Gm$-varieties.
The $\Gm$-actions on $Z$ and $W$ are trivial.

\item
The morphism $a$ and $b$ are complementary closed and open embeddings, respectively.
The varieties $W$ and $Z$   support  effective Cartier divisors in $\ov{\hm}/\SP$ and in $\ov{\ha}/\SP$, respectively.

\item
The fibers of $\ov{\ha}/\SP$ are non-canonically $\Gm$-equivariantly isomorphic to weighted projective spaces (cf. Remark
\ref{ig}).

\item
Assume that  $v_{\hm}: \hm \to \SP$ and $v_{\ha}: \ha \to \SP$  are \underline{smooth} (see Remark \ref{sm}).  Then:

\ben
\item
$\ov{\hm}, \ov{\ha}, Z$ and $W$ 
 are orbifold fibrations over $\SP$
(:= the fibers are orbifolds); 

if, in addition, $\SP$ is nonsingular,  then $\ov{\hm},  \ov{\ha}, Z$ and  $W$ are orbifolds.

\item
There is an augmented  commutative diagram:
\beq\la{sicomp2}
\xymatrix{
\w{Z} \ar[d]^-r      & \w{\ov{\hm}}  \ar[d]^-r  & 
\\
Z  \ar[d]^-{\hh} \ar[r]^-a &    \ov{\hm}   \ar[d]^-{\hh}  &    \hm  \ar[d]^-{\hh} \ar[l]_-b 
\\
W   \ar[r]^-a &  \ov{\ha}  &      \ha , \ar[l]_-b
}
\eeq
where the morphisms $r$ are  ``resolution of singularities over $\SP$" of $Z$ and $(\ov{\hm}, \hm)$ respectively, in the sense 
that $v_{\w{Z}}: \w{Z} \to \SP$ and $v_{\w{\ov{\hm}}}: \w{\ov{\hm}} \to \SP$ are smooth and projective, $r: 
\w{\ov{\hm}}\to \ov{\hm}$ is an isomorphism over $\hm$, 
and, for every $s \in \SP$, the morphisms $r_s: \w{Z}_s \to Z_s$ and $r_s: \w{\ov{\hm}}_s \to \ov{\hm}_s $ are birational (hence resolution of singularities)\footnote{in general it is not possible to resolve, say, $Z$ and at the same time resolve all the fibers of $Z/\SP$};  the boundary $\w{\ov{\hm}} \setminus \hm$ is a simple normal crossing divisor in $\w{\ov{\hm}}$ 
over $\SP.$ 
\een

\een
\end{tm}

\begin{rmk}\la{other}{\rm ({\bf Relation to earlier work.)} 

The paper \ci{si} provides a compactification of the de Rham moduli space for  a smooth projective family $X/S$ and a reductive algebraic $G$;
to my knowledge, the projectivity of this compactification is unknown. The proof of Theorem \ref{sicomp} is an adaptation of   Simpson's construction from the de Rham case to the Dolbeault case.

The paper \ci[\S4]{markman} provides a modular  compactification in the set-up or Remark \ref{sm}.

The paper \ci{ha} provides a projective compactification of the Hitchin morphism in the case of a curve of genus at least two and $G=SL_2$, via the method
of symplectic cuts.
In the same set-up, the paper \ci{edgr} provides a birational projective model for the Dolbeault moduli space in the case $G=GL_n.$ 

 \ci[Theorem 1.3.1 and (1.2.3)]{havi} provides a compactification of the Dolbeault moduli space over $S={\rm Spec} (\comp)$ in the 
  twisted case (cf. Remark \ref{sm}) of a curve of genus at least two with  $G=GL_n, SL_n, PGL_n$. In connection with this construction, the reader can consult \ci{edgr}.

The paper \ci{schmitt} provides a projective and modular compactification of  the moduli of Hitchin pairs on a projective manifold. Roughly speaking, Schmitt's Hitchin pairs on a projective manifold share the same definition
as Higgs pairs for $\G=GL_n$, except that they are not subject to the integrability condition
$\varphi \wedge \varphi =0$ on the Higgs field $\varphi$, which is automatically satisfied on curves, but is an actual condition in higher dimensions.
Note that this implies that, unlike the case of curves and $GL_n$, in higher dimension the Hitchin morphism is often not surjective.
The reader is referred to the preprint \ci{chng} for a study of the Hitchin morphism in higher dimensions.
On a curve, and for $G=GL_n$, the two are closely related: Schmitt's compactification   compactifies Simpson's Dolbeault moduli space. In higher dimensions, Schmitt's compactification contains a compactification of Simpson's as a Zariski closed subvariety.
That Schmitt's compactification should coincide with the one provided by this paper (and by  \ci{havi}'s)  should coincide for curves and $G=GL_n$, was suggested to us by Leticia Brambila Paz.
In \S\ref{asco}, following a suggestion  of A. Schmitt, we identify, the two compactifications, Schmitt's \ci[Theorem 7.1]{schmitt} and the one of Theorem \ref{sicomp},
when $X$ is a curve of genus at least two, and we consider Higgs bundles for the group $GL_n$, with degree coprime to the rank.
Even in this case, the compactification given in this paper, while not modular in nature, has interesting features, e.g. Theorem \ref{sicomp}.(6), that do not seem readily affordable via the methods in \ci{schmitt}.
}
\end{rmk}

\begin{rmk}\la{H moduli space} {\rm {\bf ($\Lambda$-modules of polynomial type)} }
The main result of this paper, namely Theorem \ref{sicomp}, parts 1-5,  is stated and proved in the case of Dolbeault moduli spaces for $\G$ reductive
algebraic.   These results hold,
with essentially the same proofs,
for moduli spaces of Higgs sheaves with fixed Hilbert polynomial and, more generally, for moduli spaces of $\Lambda$-modules
with fixed Hilbert polynomial for $\Lambda$  a sheaf of rings of differential operators of polynomial type (cf. \ci[\S1,2,3]{si2} for definitions and main results; 
see also \ci[Theorem 6.11]{si2},  on the properness of the Hitchin morphism, which also holds in this context, with the same proof). In fact, the construction of the compactification in \S\ref{achyt} 
relies on the properness of the Hitchin morphism, on the  rather general results on quotients by $\Gm$-actions in \S\ref{rfvtyu}, and on  the $\Gm$-linearization results in \S\ref{gm action gm linearization}, especially Corollary \ref{linearization2}; all of these results have evident valid counterparts
in case of $\Lambda$ of polynomial type.
\end{rmk}

%\newpage
\subsection{A compactification technique due to C. Simpson}\la{rfvtyu} $\;$

In this section we recall and slightly amplify Simpson's construction of suitable  compactifications given in \ci[\S11]{si}.

Let $S$ be  a variety endowed with the trivial $\Gm$-action. Let    $V$ and $V'$ be  varieties over $S,$
endowed with a $\Gm$-action covering the trivial $\Gm$-action over $S$, so that the structural morphisms $V,V'\to S$ are $\Gm$-equivariant.
Let  $V\to V'$ be a   $\Gm$-equivariant  proper $S$-morphism.

We thank Carlos Simpson for discussions relating to the following issue. We warmly thank the anonymous referee that brought this issue to the surface. \ci[\S11, Theorems 11.1 and 11.2]{si} are missing the seemingly necessary  hypothesis that there exists a $\Gm$-linearization for a  relatively ample line bundle
on $V$ (and for us in this paper, also on $V'$). In general, a $\Gm$-linearization for a given line bundle may fail to exist when the underlying variety  is not normal; see \ci[Introduction and Example 2.15]{br}. 

%In this paper, the failure of normality is the norm, as we work over an arbitrary base $S$. One needs to add as an hypothesis the existence of a $\Gm$-linearization, and then to verify it in applications.  Indeed this is what we do; see \S\ref{gm action gm linearization}.

We thus make the following additional assumption with respect to \ci[\S11]{si}. We assume that $V$ and $V'$  carry  relatively  ample line bundles with respect to the structural morphism to $\SP$, and that these line
bundles  admit  $\Gm$-linearizations; we do not require any kind of compatibility  between the line bundles, nor between their $\Gm$-linearizations.

\begin{tm}\la{siomp} {\rm (\ci[Thm. 11.2]{si})}
Assume that $V/S$ carries a relatively ample  line bundle admitting a $\Gm$-linearization.
Assume the fixed point set $V^{\Gm} \subseteq V$ is proper over $S,$ and that   that $0$-limits exist in $V$.
Let $U\subseteq V$ be the subset
such that the $\infty$-limits do not exist (this subset may be empty; e.g. $V/S$ proper). Then $U$ is open in $V$  and there is a universal geometric quotient $U/\Gm$. This quotient is separated and proper over $S.$
\end{tm}

Part (1,2) of 
the following proposition can be proved along the same lines of the proof of  Simpson's Theorem \ref{siomp}.
 We simply   note the following: the assumption i) on surjectivity  implies easily the assumption ii) on the fixed point set and $0$-limits.
One applies Simpson's technique to $V$ and to $V'$ to find the universal geometric
quotients.  The descended morphism between the universal geometric quotient arises  from  the equivariance of the morphism $U\to U'.$
The properness and separateness over $S$  are proved by Simpson. The properness of the descended morphism follows from
the properness of $(U/\Gm)/S$ (all our morphisms of schemes are separated and of finite type).  What needs proof is part (3).

\begin{pr}\la{sicompv}
Assume that $V/S$ and $V'/S$ carry  relatively  ample line bundles admitting $\Gm$-linearizations.
Assume the fixed point set $V^{\Gm} \subseteq V$ is proper over $S.$ Assume that $0$-limits exist in $V.$ Assume that either: i) the $\Gm$-equivariant proper $S$-morphism $V\to V'$ is surjective, or ii) the fixed point set 
${V'}^{\Gm} \subseteq V'$ is proper over $S$ 
and the $0$-limits exist in $V'.$ 
Let $U\subseteq V$ ($U' \subseteq V'$, resp.)  be the subset
such that the $\infty$-limits do not exist.  Then: 

\ben
\item
$U$  ($U'$, resp,) is open in $V$ ($V'$, resp.);  

\item
the  pre-image of $U'$ is $U;$
the proper morphism $U\to U'$ descends to a proper $S$-morphism  $U/\Gm \to U'/\Gm$  between the  geometric quotients, both of which are proper and separated over $S;$

\item
if the morphism $V\to V'$ is projective, then so is the descended morphism 
$U/\Gm  \to U'/\Gm$; if, in addition, $(U'/\Gm)/S$ is also  projective, then $(U/\Gm)/S$ is projective.

\een
\end{pr}

\begin{proof}
As it was pointed out before the statement, we only need to prove part (3). The last part of (3) follows since the composition of projective morphisms is projective (all our schemes are quasi separated and all our morphisms are quasi compact; here and in what follows we make such remarks and the reader can consult  either  EGA, or \ci{stacks-project}, to see that these are the needed conditions for the validity of our assertions).
Let $L_U$ be an$(U/\SP)$-ample line bundle admitting a $\Gm$-linearization. The $(U/\SP)$-ample  $L_U$ is automatically $(U/U')$-ample
(all our morphisms are quasi compact).
Since by definition $U$ does not contain any $\Gm$-fixed point, all the stabilizers of the $\Gm$-action on $U$ are finite, hence cyclic. 
There is an integer $\tau>0$ such that the $\Gm$-linearization on $L^{\otimes \tau}_U$ induced by the one on $L_U$
has trivial stabilizers (all our morphisms are of finite presentation and all our schemes are Noetherian).
% https://stacks.maTheoremcolumbia.edu/tag/0553
By Kempf's Descent Lemma \ci[Th\'eor\`eme 2.3]{dr-na} (the proof given there for good GIT quotients remains valid 
in the present context of a  geometric quotient), there is a line bundle $L$ on $U/\Gm$
such that it pulls-back to $L^{\otimes \tau}_U$ as a $\Gm$-bundle.

We claim that $L$ is $((U/\Gm)/(U'/\Gm))$-ample. If this were the case, then we would be done, because then the descended morphism,
being proper and quasi projective, would be projective (all our schemes are quasi compact and quasi separated).
In view of the properness of the descended morphism, 
in order to argue the desired relative ampleness, we observe that it is equivalent to the ampleness of $L$ when restricted to all closed fibers
of the descended morphism (cf. EGA III1, 4.7.1). On the other hand, given any such fiber $F_{[u']}$ over  a point $[u'] \in U'/\Gm$ with representative $u'\in U'$, the fiber $U_{u'}$
of $U\to U'$ over  $u'$  maps finitely and surjectively
onto $F_{[u']}$, this map being the quotient by the finite $\Gm$-stabilizer of $u'$. The restriction of the $(U/U')$-ample $L_U$ to 
$U_{u'}$ is automatically ample. It follows that the pull-back of $L_{|F_{[u']}}$ via the finite surjective morphism $U_{u'} \to F_{[u']}$
is ample, so that $L_{|F_{[u']}}$ is ample as well (see \cite[\href{https://stacks.maTheoremcolumbia.edu/tag/0B5V}{Tag 0B5V}]{stacks-project}).
As seen above, this implies  our claim that $L$ is $((U/\Gm)/(U'/\Gm))$-ample.
\end{proof}

\subsection{$\Gm$-action on $M/S$ and $\Gm$-linearizability of $L_M$}\la{gm action gm linearization}$\;$

Let $X/S$ be smooth and projective with structural morphism $f:X\to S$, let $\m{O}_X (1)$ be an $(X/S)$-very ample line bundle on $X$, let $P$
be a Hilbert polynomial, let $H$ be a vector bundle of finite rank on $X$, let $\Lambda:= {\rm Sym}^\bullet (H^\vee)$ be the symmetric $\m{O}_X$-algebra associated with $H^\vee$. In terms of triples $(H,\delta, \gamma)$ as in \ci[p.82]{si1}, here we set $\delta=0$ and $\gamma$ as in 
\ci[p.85]{si1}. We thus view $\Lambda$ as a sheaf of rings of differential operators of polynomial type on $X/S$.
In what follows, given a scheme $S'/S$, a superscript $-'$ denotes an object after the base change $S'/S$, e.g. $X'/S'$,  $\m{O}_{X'}(1)$, etc.

Simpson constructs the coarse moduli space $M/S$ that universally corepresents the functor $M^\#$ that assigns to $S'/S$ the set of isomorphism classes
of $p$-semistable $\Lambda$-modules on $X'$ with Hilbert polynomial $P$ wrt to $\m{O}_{X'}(1)$
(more precisely,  on the closed fibers $X'_s$ wrt to $\m{O}_{X'_s}(1)$). This moduli space comes equipped with an $(M/S)$-ample line bundle. The case of Higgs bundles is the  special case $H:= \Omega^1_{X/S}$. The Dolbeault moduli space is the even more special case discussed
in \ci[p.16]{si2}.
Next,  let us summarize the construction of $M/S$. For details we refer the reader to \ci[\S1,2,3]{si1}.
There are integers $N,m>0$ such that what follows holds.

The polynomial type $\Lambda={\rm Sym}^\bullet_{\m{O}_X}(H^\vee)$ is $\zed^{\geq 0}$-graded. The group $\Gm$ acts on $\Lambda$ by setting the action of $t\in \Gm$
to be multiplication by $t^i$ in degree $i$.
The part $\Lambda_1=\m{O}_X \oplus H^\vee$ inherits the grading in degrees zero and one, and a $\Gm$-action. 
Let $V=\comp^{P(N)}$. Simpson constructs the moduli space $M/S$ as a good quotient (terminology defined in \ci[p.61]{si1}) $Q/\!/SL(V)$ of a quasi projective variety $Q/S$ on which
the special linear group  $SL(V)$ acts, and which is equipped with a $(Q/S)$-very ample line bundle $L_Q$.
A suitable power $L_Q^{\otimes \nu}$ of $L_Q$ descends via the good quotient by $SL(V)$ to an $(M/S)$-ample line bundle on $M$.
If $S'/S$ is a variety over $S$, then an $S'$-point of $Q$ is an equivalence class  $[q]$ of quotients $q: V\otimes_\comp \Lambda_1(-N) 
\to \m{E}$ on $X'$, where two quotients are identified if they have the same kernel. The fiber of $L_Q$ over this $S'$-point is
the line bundle $\det f'_*\m{E}(m)$.
Let $\underline{u}=(u_0,u_1)$ be a graded automorphism of $V\otimes_\comp \Lambda_1(-N)$. It defines a new equivalence class
of quotients $[q(\underline{u})]$, by taking $q(\underline{u}): V\otimes_\comp \Lambda_1(-N) 
\to \m{E}(\underline{u})$, where $q(u)$ is the quotient of $V\otimes_\comp \Lambda_1(-N) $ by $\underline{u} (\ke{\,q})$.
Then $\underline{u}$ yields an isomorphism $\tilde{\underline{u}}: \det f'_* \m{E}(m) \stackrel{\simeq}\to \det f'_* \m{E} (\underline{u})(m)$.
Let $g \in SL(V)$. Define $\underline{g}=(g_0,g_1)$, by setting $g_0 = g \otimes 1_{\m{O}(-N)}$ and 
$g_1= g \otimes 1_{H^\vee(-N)}$. Then $g\cdot  [q]:= [q (\underline{u})]$ defines an $SL(V)$-action on $Q$. The construction  $\tilde{\underline{g}}$
defines an $SL(V)$-linearization of $L_Q$, i.e. an $SL(V)$-action on the total space  $T(L_Q)$  of the line bundle $L_Q$ that lifts the
$SL(V)$-action on $Q$. Given $t \in G_m$, we set $\underline{t}=(t_0,t_1)$, with $t_0 = 1_V \otimes 1_{\m{O}(-N)}$ and $t_1:
1_V \otimes t1_{H^\vee(-N)}$. By repeating what above, we obtain a $\Gm$-action on $Q$ and a lift of it to $L_Q$.
Clearly,, since $\underline{g} \circ \underline{t}=\underline{t} \circ \underline{g}$, the two actions commute with each other, and we get an action
of $SL(V)\times \Gm$ on $Q$ and an $SL(V)\times \Gm$-linearization of $L_Q$.
%A suitable power $L_Q^{\otimes \nu}$ of $L_Q$ descends via the good quotient via $SL(V)$ to an $(M/S)$-ample line bundle on $M$.
If we repeat the good quotient construction $Q/\!/SL(V)$ by replacing $Q$ with $T(L_Q^{\otimes \nu})$ and $L_Q$ with its pullback to 
$T(L_Q^{\otimes \nu})$, we see that the total space $T(L_M)$ is the corresponding good quotient by $SL(V)$ of 
$T(L_Q^{\otimes \nu})$. Via the universal property of good quotients (they corepresent the appropriate quotient functor),
since the $\Gm$-actions commutes with the $SL(V)$-action, the $\Gm$-action on $T(L_Q^{\otimes \nu})$ descends to a $\Gm$-action on $T(L_M)$,
so that $L_M$ finds itself $\Gm$-linearized. 

We have thus proved the following

\begin{pr} {\rm {\bf{($\Gm$-linearization of $L_M$ on $M$)}} }
The $SL(V)$ and $\Gm$-linearized  $(Q/S)$-very ample line bundle $L_Q$ on $Q$ descends to an $(M/S)$-ample  line bundle
$L_M$ on $M$ which inherits a $\Gm$-linearization from the one of $L_Q$.
\end{pr}

Recalling that:
\ben
\item
moduli of $p$-semistable Higgs bundles on $X/S$ with given Hilbert polynomial with respect to a given $\m{O}_X(1)$
are the special case $H=\Omega^1_{X/S}$, and that Dolbeault moduli spaces are unions of connected components of
 Higgs moduli spaces with a special polynomial (\ci[p.16, bottom]{si2});
 \item diagram (\ref{cdhm}) is $\Gm$-equivariant (cf. \S\ref{dms22}), so that we can pull back the $\Gm$-linearized  
$(M/S)$-ample  line bundle $L_{M(GL_n)}$ via $\iota_M$, and thus obtain a $\Gm$-linearized  $(M(\G)/S)$-ample 
line bundle $L_{M(\G)}$;
 \een
 we immediately obtain the following two corollaries.

\begin{cor}\la{linearization} {\rm {\bf ($\Gm$-linearized $M/S$-ample line bundle)}}
 The Dolbeault moduli space  $M/S$ for a reductive algebraic $\G$, endowed with the classical $\Gm$-action (scalar multiplication of Higgs field) admits an $(M/S)$-ample line bundle $L_M$ endowed with  a $\Gm$-linearization.  The same is true for  moduli spaces of $p$-semistable
 Higgs sheaves with a fixed Hilbert polynomial. The same is true, in the 
  case of $\Lambda$ of polynomial type,
 for moduli spaces of $p$-semistable  $\Lambda$-modules with a fixed Hilbert polynomial.
 \end{cor}
 
 \begin{cor}\la{linearization2} {\rm {\bf ($\Gm$-linearized $((M\times \mathbb A^1)/S)$-ample line bundle)}}
 Let $M/S$ be any of the three moduli spaces in Corollary \ref{linearization}.
 The pull-back $L_{M\times \mathbb A^1}$ of the line bundle $L$ to $M\times \mathbb A^1$ via the $\Gm$-equivariant projection morphism
 $M\times \mathbb A^1 \to M$ (see the forthcoming \S\ref{achyt})  is $((M\times \mathbb A^1)/S)$-ample and it inherits via pull-back
 the $\Gm$-linearization of $L_M$ given by Proposition \ref{linearization}.  The analogous conclusions hold for  the restriction $L_U$ 
 of the line bundle $L_{M\times \mathbb A^1}$ to the open
 $\Gm$-invariant subvariety $U \subseteq M\times \mathbb A^1$ (cf. \S\ref{achyt}).
 \end{cor}

%\newpage
\subsection{Construction of the  compactification of Dolbeault moduli spaces}\la{achyt}$\;$

In this subsection, we
 use Proposition \ref{sicompv}, and its notation, to  construct the desired  relative compactification $\overline{M}/S \to \overline{A}/S$ as in diagram
 (\ref{sicomp1}).  The proof
 of Theorem \ref{sicomp} concerning the properties of this construction, can be found in \S\ref{dcv}.

In what follows, we let $0 \in \mathbb A^1$ be the origin on the affine line, and, for every $s\in S$,  we let $o_s \in A_s,$ be the distinguished
(the unique $\Gm$-fixed) point in $A_s$ and we denote by $o:S\to A$ the corresponding section of $A\to S.$ Let $M_o:= S\times_AM \subseteq M$
be the $S$-subvariety of $M$ union of all fibers $M_{o_s} := h_s^{-1} (o_s) \subseteq M_s$ (the nilpotent cone for each $s$).

We let $V:=M \times \mathbb A^1,$  with the $\Gm$-action defined by setting
$t\cdot (m,x):= (t\cdot m, tx).$ The $\Gm$-fixed point set in $V$  sits inside $M=M\times \{0\}$ and coincides
with the $\Gm$-fixed-point set on $M.$
It is immediate to verify that, in this situation,  due to the properness of the Hitchin morphism, we have that $U= (M \times \mathbb A^1)  \setminus ((M \setminus M_o) \times \{0\}).$

We now repeat what above, by replacing $M$ by $A.$
Let $V':= A \times \mathbb A^1,$ endowed with the action
$t\cdot (a,x):= (t\cdot a, tx).$   The $\Gm$-fixed point set in $V'$  sits inside $A=A\times \{0\}$ and coincides
with the $\Gm$-fixed-point set on $A$, which, in turn, is the image $O$ of the  $S$-section $O:S\to A$ given by the 
``origins" $o_s \in A_s,$ $\forall s \in S.$ 
It is immediate to verify that, in this situation, 
 $U'= (A \times \mathbb A^1)  \setminus ((A \setminus O) \times \{0\}).$ 
 
Let $U_0$ be the $S$-variety fiber of the evident morphism $U\to  \mathbb A^1$ over $0\in \mathbb A^1,$ and similarly for $U'_0.$
Let $U^*$ be the $S$-variety pre-image in $U$ of $\Gm \subseteq \mathbb A^1,$ and 
similarly for $U'^*.$ Let $(a:U_0 \to U \leftarrow U^*: b)$ and $(a:U_0 \to U \leftarrow U^*: b)$ be the resulting complementary closed and 
and open  embeddings.
We have the following commutative diagram of DM stacks (all stabilizers are finite cyclic) over $S$:

\beq\la{ret5}
\xymatrix{
&&U_0  \ar[r]^-a  \ar[d]^-h  \ar[ddr] &U \ar[d]^-h \ar[ddr]   &U^*  \ar[d]^-h \ar[l]_-b \ar[ddr] &&
\\
&& U_0'  \ar[r]^-a \ar[ddr]  &U'  \ar[ddr]  &U'^*  \ar[l]_-b \ar[ddr]  &&
\\
&&&[U_0/\Gm] \ar[r]^-a  \ar[d]^-h  \ar[ddr] &  [U/\Gm] \ar[d]^-h  \ar[ddr] & [U^*/\Gm] \ar[d]^-h  \ar[l]_-b \ar[ddr] &
\\
&&&[U'_0/\Gm] \ar[r]^-a  \ar[ddr] &  [U'/\Gm]  \ar[ddr] & [U'^*/\Gm] \ar[l]_-b \ar[ddr] &
\\
Z \ar[r]^-a  \ar[d]^h &\ov{M} \ar[d]^h & M  \ar[d]^h   \ar[l]_-b &:=& U_0/\Gm \ar[r]^-a \ar[d]^-h   &  U/\Gm  \ar[d]^-h   &  U^*/\Gm  \ar[d]^-h  \ar[l]_-b 
\\
W \ar[r]^-a & \ov{A} & A  \ar[l]_-b&:=& U'_0/\Gm  \ar[r]^-a &  U'/\Gm  & U'^*/\Gm,  \ar[l]_-b
}
\eeq
where each of the six  rectangles with arrows $ha=ah$ and $hb=bh$  is Cartesian.  

Note that, by construction, it is clear that $M=U^*/\Gm$ and $A=U'^*/\Gm$.
The two adjacent rectangles on the  bottom l.h.s. are defined
to be the two adjacent rectangles on the bottom r.h.s.

%\newpage
\subsection{A $\Gm$-variation on Luna slice theorem }\la{lst}$\;$

We need the following seemingly standard result in the proof of Theorem \ref{sicomp}.(6). We thank M. Brion for pointing it out to us.

\begin{lm}\la{orb}
{\rm ({\bf Good orbifold charts})}
Let $X$ be an integral  normal  $\Gm$-variety with finite stabilizers such that the geometric quotient $Y:= X/\Gm$ exists and is separated. Then
for every point $x \in X,$  and with image $y \in Y,$
there exists a $\Gm$-stable  affine neighborhood $U_x$ of $x$ in $X,$ an affine neighborhood $V_y$ of $y$ in $Y$
and a commutative diagram: (the top horizontal morphism is induced by the  $\Gm$-action):
\beq\la{mb}
\xymatrix{
\Gm \times^{\Gamma_x} N_x  \ar[r]^-\simeq \ar[d]_-{/\Gm}  & U_x \ar[d]^-{/\Gm} \\
N_x/\Gamma_x \ar[r]^-\simeq & V_y,
}
\eeq
exhibiting $V_y$ as the geometric quotient of $U_x,$ and 
where  $\Gamma_x\subseteq \Gm$ is the stabilizer of $x$ and $N_x\subseteq U_x$ is a $\Gamma_x$-stable closed integral affine --nonsingular if $X$ is nonsingular--
subvariety  of $U_x.$
\end{lm}
\begin{proof} We limit ourselves to constructing $U_x$ and $N_x,$ leaving the remaining standard details to the reader.
By a theorem of H. Sumihiro's \ci[Thm.1]{su}, $X$ is covered by $\Gm$-invariant open affine subvarieties.
Every such subvariety, call it still $X$, admits a closed $\Gm$-equivariant embedding into a vector space with a linear action:
choose a finite dimensional vector subspace $W \subseteq \comp [X]$ of the coordinate ring of
$X$ which is $\Gm$-stable and generates the $\comp$-algebra $\comp [X]$; then
the corresponding map from $X$ to the dual of $W$ is the
desired embedding.
By considering such an embedding, we are reduced to the case where $X=V$ is a finite dimensional vector space
endowed with a linear $\Gm$-action. Let $V=\oplus_i V_{i}$ be a weight decomposition with weights $n_i.$ 
Let $d:={\rm g.c.d.}\{n_i\}.$ Let $\sum_i a_i n_i=d$ be any linear combination of the weights yielding $d$, subject to 
$a_i\neq0,$ $\forall i.$  
%Note that ${\rm g.c.d} \{a_i\}= 1.$ 
Since $\Gamma_x$ is assumed to be finite, $x\neq 0\in V.$ Let us first assume that  $x \in V$ is not on any coordinate hyperplane $H_i$ (span of the  $V_j$'s 
with $j\neq i$). Then $x=\sum v_i$ for a unique collection $v_i \in V_i$ with $v_i \neq 0,$ $ \forall i.$
Set $U_x:= V\setminus \cup_i H_i$:  it is a $\Gm$-invariant open affine neighborhood of $x.$ The $v_i$ form a basis of $\Gm$-eigenvectors for
$V.$ We define a function $f: U_x \to \Gm$ by sending a vector $u=\sum_i  z_i v_i \mapsto \prod_i z_i^{a_i}.$ The function
$f$ is  $\Gm$-equivariant, provided we endow the target $\Gm$ with the standard weight $d$ $\Gm$-action. Note that $f(x)=1.$ Set $N_x:=
f^{-1}(1).$ If $x$ lies in any multiple intersection of coordinate hyperplanes, we first project to such multiple intersection and then repeat the argument given above.
\end{proof}

%\newpage

\subsection{Proof of  the compactification Theorem \ref{sicomp}}\la{dcv}$\;$

\begin{proof}
The desired Cartesian diagram dwells in the bottom l.h.s. corner of (\ref{ret5}).

Statement (3) on $\Gm$-equivariance is clear by construction.

Statement (4) concerning the morphisms $a$ and $b$ is also clear by construction. 
Since $W$ is the divisor at infinity of a relative projective completion of a  cone, it is Cartier (cf. \ci[Appendix B5]{fu}).
The same is true for its pre-image $Z$.

%$W$ is Cartier is a general property of The part concerning the divisors $W$ and $Z$  follows from the fact that they arise in connection with the Cartier divisors $M\times \{0\}$ and $A\times \{0\},$ respectively, which are defined by $\Gm$-invariant functions, and from the fact that at general points of $M$ and of $A$, the stabilizers are trivial.

Statement (5) is clear by construction.

We now prove statement (2)  to the effect that the morphisms $h$ are projective. It is enough to show that $\overline{M}/\overline{A}$ is projective and,
in view of the fact that $\overline{M}/\overline{A}$ is proper, it is enough to produce an $(\overline{M}/\overline{A})$-ample
line bundle (our schemes are quasi compact and quasi separated). In order to do so, we apply Corollary \ref{linearization2} and  Proposition \ref{sicompv}.(3).

We now prove statement (1) to the effect that the  structural morphisms for $\ov{M}, \ov{A}, W$ and $Z$ over $S$ are projective.
It is clear that the  ``relative weighted projective space" $\ov{A}/S$ is projective, as it is the Proj of a suitable graded $\m{O}_S$-algebra
associated with the symmetric $\m{O}_S$-algebra giving $A/S$; see  \S\ref{dms22}. Since $\ov{M}/\ov{A}$ is projective
by assertion (2), we have that the compositum $\ov{M}/S$ is projective (all our schemes are quasi compact and quasi separated). 
% \ci[\href{https:/\!/stacks.maTheoremcolumbia.edu/tag/01W7}{Tag 01W7, Lemma 28.41.14}]{stacks-project}. 
It follows that $W/S$ and $Z/S$ are projective as well.

As to assertion (6):
part (a) follows from  Lemma \ref{orb};
part (b) follows from  \ci[Thm. 1.2]{be}.
\end{proof}

%\newpage
\subsection{Comparison with A. Schmitt's Compactification}\la{asco}$\;$

The goal of this section is to observe that in the special case mentioned in Remark \ref{sm}, i.e.  when $X$ is a nonsingular projective curve and we take $GL_n$
Higgs bundles of degree coprime to the rank, then the compactification constructed in Theorem \ref{sicomp}, coincides
with the corresponding moduli of Hitchin pairs constructed by A. Schmitt in \ci{schmitt}. We thank A. Schmitt for providing us 
with the sketch of the needed argument; see the proof of Proposition \ref{ran}. It seems likely that the two compactifications coincide more generally for (untwisted) Dolbeault moduli space  of families of projective manifolds of any dimension; we have not verified this.

Let $X/\comp$ be a nonsingular projective manifold,   let $\mathcal O_X(1)$ be  an ample line bundle on $X$, let $L$ be a line bundle on 
$X$ and let $P$ be a polynomial. 

In the paper \ci{schmitt}, A. Schmitt  introduced the notion of Hitchin  pairs  $(E,\e, \varphi)$ of type $(P,L)$ on $X$:
$E$ is a torsion-free coherent sheaf on $X$, $\varphi: E \to E\otimes L$ is a twisted endomorphism, $\e \in \comp$, and $P$ is the Hilbert polynomial of $(E, \m{O}_X(1))$. 

Note that in the definition  of an Hitchin pair, the twisted endomorphism $\varphi$ is not subject to the
Higgs/Simpson-type vanishing condition $\varphi \wedge \varphi =0$; in particular, the 
 $(E,\varphi)$-component of an Hitchin pair is not necessarily an Higgs sheaf. Since the aforementioned vanishing condition is automatically satisfied when $\dim X=1$, in that case, the component $(E,\varphi)$ of an Hitchin pair yields an Higgs sheaf 
for the group $GL_n$. 

There are the notions of:  equivalent Hitchin pairs;  (semi)stable Hitchin pair; (equivalence classes of) families of Hitchin pairs over a Noetherian scheme $S$; the functors $M^{(s)s}_{L,P}$  of equivalence classes of families of (semi)stable Hitchin pairs of type $(L,P)$.

\ci[Theorem 7.1]{schmitt} shows that there is a projective variety $\m{M}^{ss}_{(L,P)}$,
whose closed points naturally correspond to certain equivalence classes (semistable Hitchin pairs with graded objects that are equivalent Hitchin pairs) of semistable Hitchin pairs of type $(L,P)$. The open subvariety $\m{M}^{s}_{(L,P)}\subseteq \m{M}^{ss}_{(L,P)}$
of stable pairs coarsely represents the functor $M^s_{(L,P)}$.

There is the natural  $\Gm$-action on $\m{M}:= \m{M}^{ss}_{(L,P)}$ given by scalar multiplication  on $\varphi$.
The fixed-point set is the union of: the part that corresponds to semistable Hitchin pairs with $\varphi=0$ (in which case, we must have
$\e \neq 0$, by the very  definition of stability of Hitchin pairs), i.e. the Gieseker moduli space;  the part $\m{M}_\infty$
which corresponds to $\e=0$.

If  we denote by $\m{M}_{\neq 0}$   the $\Gm$-invariant open subvariety corresponding to $\e\neq 0$, then $\m{M}_\infty = \m{M}_{\neq 0} /\!/ \Gm$. 

In the remainder of this section, we place ourselves in the  situation of Remark \ref{sm}:
$GL_n$-Higgs bundles over a projective  connected nonsingular curve  $X$ of genus $g(X) \geq 2$, 
of degree coprime to the rank, and the line bundle $L$ is 
 either the canonical bundle of $X$, or
any fixed line bundle of degree bigger that $2g(X)-2$. 

Then the corresponding Dolbeault Simpson moduli space $M$ coincides with Schmitt's moduli space of Hitchin pairs  $\m{M}_{\neq 0}$,
and in either case semistability coincides with stability (due to the coprimality condition).
There are  a natural proper Hitchin morphism for both moduli spaces and they coincide.

The compactification Theorem \ref{sicomp}.(6) applies to $M$ and we obtain the  compactification $M\subseteq \ov{M}$, with boundary $Z=(M\setminus M_o)/\Gm= M/\!/\Gm$  (cf. \S\ref{achyt}).

A. Schmitt has informed us that, as one may expect, one should have a natural  $\Gm$-equivariant identification
of $(M,Z)$ with $(\m{M}, \m{M}_\infty)$. The resulting identification identifies the corresponding Hitchin morphisms.
See Proposition \ref{ran}. This identification is not used in this paper.

Schmitt's construction of the compactification  is modular (i.e. it provides a modular interpretation of the boundary).
The compactification provided by Theorem \ref{sicomp} has the following extra features: it allows us to prove 
Theorem \ref{sicomp}.(6) and the upcoming Proposition \ref{vcb1},  and is valid for all (families of) projective manifolds and all reductive algebraic groups. 

We thank A. Schmitt for the proof of the following
\begin{pr}\la{ran}
The two compactifications $\ov{M}$ and $\m{M}$ coincide. The identification is $\Gm$-equivariant and the Hitchin morphisms
correspond.
\end{pr}
\begin{proof}

The Simpson moduli space $M$ is a GIT quotient of some parameter space $R$ (see \ci{si1,si2}) which, due to the fact that stability coincides with
semistability, admits a universal family $(E, \varphi)$ of stable  Higgs bundles over it. This family gives rise to   a family
$(E, \e, \varphi, \m{O})$
of  Hitchin pairs over $R \times \comp$ in the sense of \ci{schmitt}. The Hitchin pairs in question are automatically stable over $R\times \comp^*$, but, in order to have stability, one needs to remove from $R\times \{0\}$ the closed subset  where the twisted endomorphisms are nilpotent.

Since stability and semistability coincide,  $\m{M}$ is a coarse moduli space for the functor and, by what above, there is
the classifying morphism $U\to \m{M}$, where $U$ is a suitable open subset of $R\times \comp$. By construction, the morphism
factors through $\ov{M}$, hence a natural morphism $\ov{M} \to \m{M}$, which identifies the two open subsets  $\m{M}_{\neq 0}$
and $M$.

Let $(E, \e, \varphi, N)$ be a family of semistable Hitchin pairs over a Noetherian scheme $S$ (recall that $N$ is a line bundle on $S$ and $\e \in \Gamma (S, N)$).

Let $\{U_i\}_I$ be an open covering of $S$ over which $N$ can be trivialized. By restricting $(E, \varphi)$ over $U_i \times X$, we obtain morphisms $U_i \to M$. By restricting $\e$ and using the trivializations, we obtain morphisms $U_i \to \comp$. We thus get morphisms $U_i \to M \times \comp$. By the definition of semistability of Hitchin pairs (no nilpotent fields are allowed), the image
of such morphisms must lie in the complement of $M_o \times \{0\}$.  We thus obtain morphisms $U_i \to \ov{M} =
M\times \comp /\!/ \Gm$. These morphisms glue and  yield a morphism $S \to \ov{M}$. By the universal property
of $\m{M}$ (cf. \ci[Theorem 7.1.i)]{schmitt}, we obtain a morphism $\ov{M} \to \m{M}$. This morphism also 
identifies the two open subsets $M$ and $\m{M}_{\neq 0}$.

Note that, in general, $M$ is dense in $\ov{M}$ and that $\m{M}_{\neq 0}$ is dense in $\m{M}$; in fact, in the current situation,
$\ov{M}$ and $\m{M}$ are in fact irreducible.
It follows that the two morphisms $\ov{M} \to \m{M}$ and $\m{M} \to \ov{M}$ obtained above, are inverse to each other.
They clearly are $\Gm$-equivariant and also identify $Z$ with $\m{M}_\infty$. The Hitchin morphism are already identified
on  the open sets $M$ and $\m{M}_{\neq 0}$, hence they are identified after the compactifications.
\end{proof}

%\newpage
\subsection{Additional properties of the compactification (\ref{sicomp2}) when $M/S$ is smooth}$\;$

While it is rare for the Dolbeault moduli spaces to be nonsingular, they are so in interesting cases;
see Remark \ref{sm}.
The following proposition summarizes some  topological properties of the compactification given by Theorem \ref{sicomp}.(6), when $M$ is smooth over $S.$

\begin{pr}\la{vcb1} Assume that $S$ is nonsingular and that $M/S$ is smooth.
Let things be as in diagram {\rm (\ref{sicomp2})}. Then we have the following properties:

\ben
\item
The varieties $\ov{M},$ $\ov{A}, Z$ and $W$ are orbifolds and,
for every $s\in S$, so are  the varieties    $\ov{M}_s,\ov{A}_s, Z_s$ and 
 and $W_s.$  
 
In particular,  for all these varieties,  up to the usual  dimensional and Tate shifts,  the intersection complexes  $IC$ and the dualizing complex $\omega$ coincide with the constant sheaf, for example:
 (we ignore the Tate shifts)
\beq\la{fvr}
IC_{\ov{M}} = \rat_{\ov{M}} [\dim \ov{M}], \quad  \omega_{\ov{M}} = \rat_{\ov{M}} [2\dim \ov{M}].
\eeq
\item
We have the following identities for extraordinary pull-backs:
\beq\la{se21}
a^!  \rat_{\ov{M}} =  \rat_{Z} [-2], \quad
i^!   \rat_{\ov{M}} =  \rat_{\ov{M}_s}[-2], \quad
 i^!   \rat_{Z} =  \rat_{Z_s}[-2], \quad
 i^! a^! \rat_{\ov{M}} = a^! i^! \rat_{\ov{M}} = \rat_{Z_s}[-4].
\eeq
so that, up to the appropriate cohomological shift, the complexes in {\rm (\ref{se21})} are perverse semisimple.

\item
Finally,  if $(S,s)$ is a nonsingular curve with a distinguished point on it, then we have the following vanishing  property for the
resulting  vanishing cycle complexes on $\ov{M}_s, \ov{Z}_s$:
\beq\la{va}
\xymatrix{
\phi_{{\ov{M}}/S} (\rat_{{\ov{M}}})=0, &\phi_{Z/S} (a^!\rat_{\ov{M}})=0.
}
\eeq
\item
Conclusions (1,2,3) holds when $M$ is the moduli of Higgs bundles over curves  with degree coprime to the rank
and group  $G=GL_n, SL_n, PGL_n$.

\een
\end{pr}

\begin{proof}
We prove (1).
The first assertion on orbifolds is Theorem \ref{sicomp} part (6a). The identities (\ref{fvr})  are standard for orbifolds.

We prove (2). The assertions (\ref{se21}) are standard as well. For example: 
\[a^! \rat_{\ov{M}}=
a^! \omega_{\ov{M}}[-2\dim \ov{M}]=  \omega_{Z}[-2\dim \ov{M}] = \rat_{Z}[-2].
\]

We prove (3).
We prove the vanishing assertion  (\ref{va})   for $\ov{M}$.  The one for $Z$ can be proved in the same way.
Since $\w{\ov{M}}/S$ is smooth, we have $\phi_{\w{\ov{M}}}(\rat_{\w{\ov{M}}})=0$. Since $r$ is proper, we have
$\phi_{\ov{M}}(r_*\rat_{\w{\ov{M}}}) = r_* \phi_{\w{\ov{M}}}(\rat_{\w{\ov{M}}})=0$. It remains to show that $\rat_{\ov{M}}$ is a direct summand
of $r_* \rat_{\w{\ov{M}}}$. This follows from the decomposition theorem \ci{bbd}. Given the special orbifold situation, this can also be seen as follows. Consider the adjunction morphisms:
\beq\la{40}
\xymatrix{
r_!r^!  \rat_{\ov{M}} \ar[r]^-x & \rat_{\ov{M}}  \ar[r]^-y & r_*r^*  \rat_{\ov{M}}}
\eeq
Since the dualizing sheaves are constant shifted,   in view of  the identity $\omega_A = g^!\ \omega B$ (valid
for every morphism of varieties $g:A \to B$), 
and in view of the properness of $r$ (so that $r_*=r_!$), we may re-write
(\ref{40}) as:
\beq\la{42}
\xymatrix{
r_*\rat_{\w{\ov{M}}} \ar[r]^x &   \rat_{\ov{M}}  \ar[r]^y & r_* \rat_{\w{\ov{M}}}
}
\eeq
Since $r$ is a resolution, the endomorphism $x\circ y: \rat_{\ov{M}} \to \rat_{\ov{M}}$ can be viewed
as the  identity on a dense open subset, hence it is the identity on the connected  $\ov{M}$. It follows that $\rat_{\ov{M}}$
is a direct summand of $r_* \rat_{\w{\ov{M}}}$, as predicated.

We prove (4). The case when $G=GL_n, SL_n$ is covered by what above because then $M/\SP$ is smooth.
The case when $G=PGL_n$ follows easily from the case when $G=SL_n$ because the whole picture for $PGL_n$ is the quotient of the whole picture for $SL_n$ by the finite group scheme over $\SP$ of $n$-torsion points in the relative Jacobian
of the family of curves.
\end{proof}

\begin{rmk}\la{fjfj}
Proposition \ref{vcb1}.(3) can be used to study the long exact sequence of cohomology
of the triple $(Z, \ov{M}, M)$ and generalize,  by means of (\ref{va}),  the main result in \ci{dema}
in the context of Remark \ref{sm} as follows: the long exact sequence in relative cohomology for the triple $(Z, \ov{M}, M)$  takes the form of a long exact sequence of filtered vector spaces
$\ldots \to (H^{*-2}(Z),P) \to (H^*(\ov{M}),P) \to (H^*(M),P) \to \ldots$, where $P$ stands for
the appropriately shifted perverse Leray filtrations
 This study is carried out in greater generality in a forthcoming paper.
\end{rmk}

 The example below  points to the need of exercising caution in connection with  the vanishing assertion in  Proposition \ref{vcb1}.(\ref{va}). 
 \begin{ex}\la{fo}
Let $v_{\w{X}}:\w{X} \stackrel{r}\to X \stackrel{v_X}\to S$ be such that:  $\w{X}/S$ is  the family proper over a disk $S$ with general member
a smooth quadric surface $F_0$  and with special member  the Hirzebruch surface $F_2$; $r$ is the birational contraction of the $(-2)$-curve
in the central fiber to a point $p$. We have $\phi_{\w{X}}(\oql)=0$, which implies $\phi_X (r_* \oql)=0$;
since $r$ is small, we have $r_*\oql=\m{IC}_X$ (the intersection complex of $X$ placed in cohomological degrees $[0,2]$),  so that
 $\phi_X (\m{IC}_X)=0$. Note however that $\phi \rat_X = \rat_p [-2]\neq 0.$
\end{ex}

%\newpage

\end{document}